\documentclass[10pt, a4paper,authoryear]{elsarticle} 
\usepackage[utf8]{inputenc}
\usepackage{amsmath}
\usepackage{amssymb}
\usepackage{natbib}

\usepackage{bm}
\usepackage[final]{showkeys}

\usepackage{epstopdf}
\usepackage{amsfonts}
\usepackage{dsfont}
\usepackage{color}
\usepackage{xcolor}

\usepackage{amsmath}
\usepackage{amsthm}
\usepackage{color}
\usepackage{xcolor}

\usepackage{lineno}
\newcommand*\patchAmsMathEnvironmentForLineno[1]{%
  \expandafter\let\csname old#1\expandafter\endcsname\csname #1\endcsname
  \expandafter\let\csname oldend#1\expandafter\endcsname\csname end#1\endcsname
  \renewenvironment{#1}%
     {\linenomath\csname old#1\endcsname}%
     {\csname oldend#1\endcsname\endlinenomath}}%
\newcommand*\patchBothAmsMathEnvironmentsForLineno[1]{%
  \patchAmsMathEnvironmentForLineno{#1}%
  \patchAmsMathEnvironmentForLineno{#1*}}%
\AtBeginDocument{%
\patchBothAmsMathEnvironmentsForLineno{equation}%
\patchBothAmsMathEnvironmentsForLineno{align}%
\patchBothAmsMathEnvironmentsForLineno{flalign}%
\patchBothAmsMathEnvironmentsForLineno{alignat}%
\patchBothAmsMathEnvironmentsForLineno{gather}%
\patchBothAmsMathEnvironmentsForLineno{multline}%
} 

\usepackage{url}
\theoremstyle{definition}
\newtheorem{definition}{Definition}%

\newtheorem{theorem}{Theorem}
\newtheorem{proposition}[theorem]{Proposition}%
\newtheorem{lemma}{Lemma}

\newtheorem{remark}{Remark}%
\newtheorem{remarks}[remark]{Remarks}
\newtheorem{example}{Example}

\numberwithin{equation}{section}

\newcommand{\sd}{\mathrm{d}}
\newcommand{\e}{\mathrm{e}}
\newcommand{\ci}{\mathrm{i}}
\newcommand\xqed[1]{%
  \leavevmode\unskip\penalty9999 \hbox{}\nobreak\hfill
  \quad\hbox{#1}}
\newcommand\triag{\xqed{$\triangle$}}

\title{Limit Theorems for the Fractional Non-homogeneous Poisson Process}

\author{$\text{Nikolai Leonenko}^1$ \& $\text{Enrico Scalas}^2$ \& $\text{Mailan Trinh}^2$\\
	\footnotesize{${}^1$Cardiff University, School of Mathematics, Senghennydd Road, Cardiff, CF24 4AG, UK}\\
	\footnotesize{${}^2$Department of Mathematics, School of Mathematical and Physical Sciences}\\
	\footnotesize{University of Sussex, Brighton, BN1 9QH, UK}}

\begin{document}
	
\begin{abstract}
The fractional non-homogeneous Poisson process was introduced by a time-change of the non-homogeneous Poisson process with the inverse $\alpha$-stable subordinator. We propose a similar definition for the (non-homogeneous) fractional compound Poisson process. We give both finite-dimensional and functional limit theorems for the fractional non-homogeneous Poisson process and the fractional compound Poisson process. The results are derived by using martingale methods, regular variation properties and Anscombe's theorem. Eventually, some of the limit results are verified in a Monte Carlo simulation.
\end{abstract}	
		

\maketitle
\section{Introduction}
\label{sec:introduction}
The (one-dimensional) homogeneous Poisson process can be defined as a renewal process by specifying the distribution of the waiting times $J_i$ to be i.i.d. and to follow an exponential distribution. The sequence of associated arrival times 
\begin{equation*}
  T_n = \sum_{i=1}^n J_i, \,\, n \in \mathbb{N},\, T_0=0,
\end{equation*}
gives a renewal process and its corresponding counting process
\begin{equation*}
  N(t) = \sup\{n: T_n \leq t\} = \sum_{n=0}^{\infty}n \mathds{1}_{\{T_n \leq t < T_{n+1}\}}
\end{equation*}
is the Poisson process with parameter $\lambda > 0$. Alternatively, $N(t)$ can be defined as a L\'evy process with stationary and Poisson distributed increments. Among other approaches, both of these representations have been used in order to introduce a fractional homogenous Poisson process (FHPP).
As a renewal process, the waiting times are chosen to be i.i.d. Mittag-Leffler distributed instead of exponentially distributed, i.e.
\begin{equation}
  \label{eq:16}
  \mathbb{P}(J_1 \leq t) = 1 - E_\alpha(-(\lambda t)^\alpha), \qquad t \geq 0 
\end{equation}
where $E_\alpha (z)$ is the one-parameter Mittag-Leffler function defined as
\begin{equation*}
E_\alpha (z) = \sum_{n=0}^\infty \frac{z^n}{\Gamma(\alpha n + 1)}, \quad z \in \mathbb{C}, \alpha \in [0,1).
\end{equation*}
The Mittag-Leffler distribution was first considered in \cite{Gnedenko_1968} and \cite{Khintchine_1969}. A comprehensive treatment of the FHPP as a renewal process can be found in \cite{Mainardi_2004} and \cite{Politi_2011}.\\
Starting from the standard Poisson process $N(t)$ as a point process, the FHPP can also be defined as $N(t)$ time-changed by the inverse $\alpha$-stable subordinator. \cite{Meerschaert_2011} showed that both the renewal and the time-change approach yield the same stochastic process (in the sense that both processes have the same finite-dimensional distribution). \cite{Laskin_2003} and \cite{Beghin_Orsingher_2009, Beghin_Orsingher_2010} derived the governing equations associated with the one-dimensional distribution of the FHPP.\\
In \cite{Leonenko_2017}, we introduced the fractional non-homogeneous Poisson process (FNPP) as a generalization of the FHPP. The non-homogeneous Poisson process is an additive process with deterministic, time dependent intensity function and thus generally does not allow a representation as a classical renewal process. However, following the construction in \cite{Gergely_1973, Gergely_1975} we can define the FNPP as a general renewal process. Following the tine-change approach, the FNPP is defined as a non-homogeneous Poisson process time-changed by the inverse $\alpha$-stable subordinator.\\
Among other results, we have discussed in our previous work that the FHPP can be seen as a Cox process. Following up on this observation, in this article, we will show that, more generally, the FNPP can be treated as a Cox process discussing the required choice of filtration. Cox processes or doubly stochastic processes (\cite{Cox_1955}, \cite{Kingman_1964}) are relevant for various applications such as filtering theory \citep{Bremaud_1981}, credit risk theory \citep{Bielecki_2002} or actuarial risk theory \citep{Grandell_1991} and, in particular, ruin theory \citep{Biard_Saussereau_2014,Biard_Saussereau_2016}. Subsequently, we are able to identify the compensator of the FNPP. A similar generalization of the original Watanabe characterization \citep{Watanabe_1964} of the Poisson process can be found in case of the FHPP in \cite{Aletti_2016}.\\
Limit theorems for Cox processes have been studied by \cite{Grandell_1976} and \cite{Serfozo_1972a, Serfozo_1972b}. Specifically for the FHPP, scaling limits have been derived in \cite{Meerschaert_2004} and discussed in the context of parameter estimation in \cite{Cahoy_2010}.\\

The rest of the article is structured as follows: In Section \ref{sec:fract-poiss-proc-1} we give a short overview of definitions and notation concerning the fractional Poisson process. Section \ref{sec:fract-comp-as} and \ref{sec:fract-poiss-proc} are devoted to the application of the Cox process theory to the fractional Poisson process which allows us to identify its compensator and thus derive limit theorems via martingale methods. A different approach to deriving asymptotics is followed in Section \ref{sec:scaling-limit} and requires a regular variation condition imposed on the rate function of the fractional Poisson process. The fractional compound process is discussed in Section \ref{sec:appl-fract-comp} where we derive both a one-dimensional limit theorem using Anscombe's theorem and a functional limit. Finally, we give a brief discussion of simulation methods of the FHPP and verify the results of some of our results in a Monte Carlo experiment.
\section{The fractional Poisson process}
\label{sec:fract-poiss-proc-1}
This section serves as a brief revision of the fractional Poisson process both in the homogeneous and the non-homogeneous case as well as a setup of notation.\\

Let $(N_1(t))_{t\geq 0}$ be a standard Poisson process with parameter $1$. Define the function
\begin{equation*}
  \Lambda(s,t) := \int_s^t \lambda(\tau) \,\sd \tau,
\end{equation*}
where $s,t \geq 0$ and $\lambda: [0,\infty) \longrightarrow (0,\infty)$ is locally integrable. For shorthand $\Lambda(t) := \Lambda(0,t)$ and we assume $\Lambda(t) \to \infty$ for $t\to \infty$. We get a non-homogeneous Poisson process $(N(t))_{t\geq 0}$, by a time-transformation of the homogeneous Poisson process with $\Lambda$: 
\begin{equation*}
  N(t) := N_1(\Lambda(t)).
\end{equation*} 
The $\alpha$-stable subordinator is a L\'evy process $(L_\alpha(t))_{t\geq 0}$ defined via the Laplace transform
\begin{equation*}
  \mathbb{E}[\exp(-uL_\alpha(t))] = \exp(-tu^\alpha).
\end{equation*}
The inverse $\alpha$-stable subordinator $(Y_\alpha(t))_{t\geq 0}$ (see e.g. \cite{Bingham_1971}) is defined by
\begin{equation*}
  Y_\alpha(t) := \inf\{ u \geq 0: L_\alpha(u) > t\}.
\end{equation*}
We assume $(Y_\alpha(t))_{t \geq 0}$ to be independent of $(N(t))_{t\geq 0}$. For $\alpha\in (0,1)$, the fractional non-homogeneous Poisson process (FNPP) $(N_\alpha(t))_{t\geq 0}$ is defined as 
\begin{equation}
  \label{eq:23}
  N_\alpha(t) := N(Y_\alpha(t)) = N_1(\Lambda(Y_\alpha(t)))
\end{equation}
 (see \cite{Leonenko_2017}). Note that the fractional homogeneous Poisson process (FHPP) is a special case of the non-homogeneous Poisson process with $\Lambda(t) = \lambda t$, where $\lambda(t) \equiv \lambda > 0$ a constant. 
Recall that the density $h_\alpha(t, \cdot)$ of $Y_\alpha(t)$ can be expressed as \citep[see e.g.][]{Meerschaert_2013, Leonenko_2015}
\begin{equation}
  \label{eq:5}
  h_\alpha(t,x) = \frac{t}{\alpha x^{1+\frac{1}{\alpha}}} g_\alpha\left( \frac{t}{x^{\frac{1}{\alpha}}}\right), \quad x\geq 0, t\geq 0,
\end{equation}
where $g_\alpha (z)$ is the density of $L_\alpha(1)$ given by
\begin{align}
  g_\alpha(z) &= \frac{1}{\pi} \sum_{k=1}^{\infty}(-1)^{k+1}\frac{\Gamma(\alpha k+1)}{k!} \frac{1}{z^{\alpha k +1}} \sin(\pi k \alpha)\nonumber
\end{align}
The Laplace transform of $h_\alpha$ can be given in terms of the Mittag-Leffler function
\begin{equation}
  \label{eq:FractionalLT_SPL:1}
  \tilde{h}_\alpha(t,y) = \int_0^\infty \e^{-xy} h_\alpha(t, x) \, \sd x = E_\alpha(-y t^\alpha),
\end{equation}
and for the FNPP the one-dimensional marginal distributions are given by
\begin{equation*}
  \mathbb{P}(N_\alpha(t) = x) = \int_0^\infty \e^{-\Lambda(u)} \frac{\Lambda(u)^x}{x!} h_\alpha(t,u) \,\sd u.
\end{equation*}
Alternatively, we can construct an non-homogeneous Poisson process as follows (see \cite{Gergely_1973}). Let $\xi_1, \xi_2, \ldots$ be a sequence of independent non-negative random variables with identical continuous distribution function
\begin{equation*}
  F(t) = \mathbb{P}(\xi_1 \leq t) = 1 - \exp(-\Lambda(t)), t \geq 0.
\end{equation*}
Define
\begin{equation*}
  \zeta'_n := \max\{\xi_1, \ldots, \xi_n\}, \quad n=1,2, \ldots
\end{equation*}
and
\begin{equation*}
  \varkappa_n = \inf\{k \in \mathbb{N}: \zeta'_k > \zeta'_{\varkappa_{n-1}}\}, \quad n = 2, 3, \ldots 
\end{equation*}
with $\varkappa_1 = 1$.
Then, let $\zeta_n := \zeta'_{\varkappa_n}$. The resulting sequence $\zeta_1, \zeta_2, \ldots$ is strictly increasing, since it is obtained from the non-decreasing sequence $\zeta'_1, \zeta'_2, \ldots$ by omitting all repeating elements. Now, we define
\begin{align*}
  N(t) &:= \sup\{k \in \mathbb{N}: \zeta_k \leq t\}= \sum_{n=0}^\infty n \mathds{1}_{\{ \zeta_n \leq t < \zeta_{n+1}\}}, \quad t \geq 0
\end{align*}
where $\zeta_0 = 0$.
By Theorem 1 in \cite{Gergely_1973}, we have that $(N(t))_{t\geq 0}$ is a non-homogeneous Poisson process with independent increments and
\begin{equation*}
  \mathbb{P}(N(t) = k) = \exp(-\Lambda(t)) \frac{\Lambda(t)^k}{k!}, \quad k = 0, 1, 2, \ldots.
\end{equation*}
It follows via the time-change approach that the FNPP can be written as
\begin{align*}
  N_\alpha(t) = \sum_{n=0}^\infty n \mathds{1}_{\{\zeta_n \leq Y_\alpha(t) < \zeta_{n+1}\}}
  \stackrel{\text{a.s.}}{=}\sum_{n=0}^\infty n \mathds{1}_{\{L_\alpha(\zeta_n) \leq t < L_\alpha(\zeta_{n+1})\}},
\end{align*}
where we have used that $L_\alpha(Y_\alpha(t)) = t$ if and only if $t$ is not a jump time of $L_\alpha$ (see \cite{Embrechts_Hofert_2013}). 

\section{The FNPP as Cox process}
\label{sec:fract-comp-as}

Cox processes go back to \cite{Cox_1955} who proposed to replace the deterministic intensity of a Poisson process by a random one. In this section, we discuss the connection between FNPP and Cox processes.

\begin{definition}
  \label{def:fpp-as-cox}
  Let $(\Omega, \mathcal{F}, \mathbb{P})$ be a probability space and
  $(N(t))_{t\geq 0}$ be a point process adapted to a filtration $(\mathcal{F}^N_t)_{t\geq 0}$. $(N(t))_{t\geq 0}$ is a Cox process if there exist a right-continuous, increasing process $(A(t))_{t\geq 0}$ such that, conditional on the filtration $(\mathcal{F}_t)_{t\geq 0}$, where
  \begin{equation}
    \label{eq:22}
    \mathcal{F}_t := \mathcal{F}_0 \vee \mathcal{F}^N_t, \quad \mathcal{F}_0 = \sigma(A(t), t\geq 0),
\end{equation}
then $(N(t))_{t\geq 0}$ is a Poisson process with intensity $\sd A(t)$.
\end{definition}
In particular we have by definition  $\mathbb{E}[N(t) \vert \mathcal{F}_t] = A(t)$ and 
\begin{equation*}
  \mathbb{P}(N(t) = k \vert \mathcal{F}_t) = \e^{-A(t)} \frac{A(t)^k}{k!}, \qquad k = 0,1,2, \ldots
\end{equation*}
\newpage
\begin{remarks}\hfill
\label{rem:fnpp-as-cox}
  \begin{enumerate}
  \item A Cox process $N$ is said to be directed by $A$, if their relation is as in the above definition. Cox processes are also called $(\mathcal{F}_t)_{t\geq 0}$-Cox process, doubly stochastic processes, conditional Poisson processes or $(\mathcal{F}_t)_{t\geq 0}$-conditional Poisson process.
  \item As both $N$ and $A$ are increasing processes, they are also of finite variation. The path $t \mapsto N(t, \omega)$ induces positive measures $\sd N$ and integration w.r.t. this measure can be understood in the sense of a Lebegue-Stieltjes integral (see p. 28 in \cite{Jacod_Shiryaev_2003}). The same holds for the paths of $A$.
  \item     Definitions vary across the literature. The above definition can be compared to essentially equivalent definitions: in \cite{Bremaud_1981}, 6.12 on p. 126 in \cite{Jacod_Shiryaev_2003}, Definition 6.2.I on p. 169 in \cite{Daley_2008}, where $\mathcal{X}=\mathbb{R}^{+}$ and  Definition 6.6.2 on p.193 in \cite{Bielecki_2002}.
  \item Cox processes find applications in credit risk modelling. In this context $A(t)$ is referred to as hazard process (see \cite{Bielecki_2002}).
  \end{enumerate}
\end{remarks}

In case of the FHPP, there exist characterizing theorems, for example, found in \cite{Yannaros_1994} and \cite{Grandell_1976} (Theorem 1 of Section 2.2). They use the fact that the FHPP is also a renewal process and allows for a characterization via the Laplace transform of the waiting time distributions. This has been worked out in detail in Section 2 in \cite{Leonenko_2017}. However, the theorem does not give any insight about the underlying filtration setting. This will become more evident from the following discussion concerning the general case of the FNPP.\\

In the non-homogeneous case, we cannot apply the theorems which characterize Cox renewal processes as the FNPP cannot be represented as a classicalrenewal process. Therefore, we need to resort to Definition \ref{def:fpp-as-cox} for verification. It can be shown that the FNPP is a Cox process under a suitably constructed filtration. We will follow the construction of doubly stochastic processes given in Section 6.6 in \cite{Bielecki_2002}. Let $(\mathcal{F}_t^{N_\alpha})_{t\geq 0}$ be the natural filtration of the FNPP $(N_\alpha(t))_{t\geq 0}$
\begin{equation*}
  \mathcal{F}_t^{N_\alpha} := \sigma(\{N_\alpha(s): s \leq t\}).
\end{equation*}
We assume the paths of the inverse $\alpha$-stable subordinator to be known, i.e.
\begin{equation}
  \label{eq:1}
  \mathcal{F}_0 := \sigma(\{Y_\alpha(t), t\geq 0\}).
\end{equation}
We refer to this choice of initial $\sigma$-algebra as \textit{non-trivial initial history} as opposed to the case of \textit{trivial initial history}, which is $\mathcal{F}_0 = \{\varnothing, \Omega\}$.\\
The overall filtration $(\mathcal{F}_t)_{t\geq 0}$ is then given by 
 \begin{equation}
   \label{eq:20}
   \mathcal{F}_t := \mathcal{F}_0 \vee \mathcal{F}_t^{N_\alpha},
 \end{equation}
which is sometimes referred to as \textit{intrinsic history}. If we choose a trivial initial history, the intrinsic history will coincide with the natural filtration of the FNPP.
 \begin{proposition}
   Let the FNPP be adapted to the filtration $(\mathcal{F}_t)$ as in (\ref{eq:20}) with non-trivial initial history $\mathcal{F}_0 := \sigma(\{Y_\alpha(t), t\geq 0\})$. Then the FNPP is a $(\mathcal{F}_t)$-Cox process directed by $(\Lambda(Y_\alpha(t)))_{t\geq 0}$. 
 \end{proposition}
 \begin{proof}
This follows from Proposition 6.6.7. on p. 195 in \cite{Bielecki_2002}. We give a similar proof:
   As $(Y_\alpha(t))_{t\geq 0}$ is $\mathcal{F}_0$-measurable we have
   \begin{align}
     \mathbb{E}[\exp&\{\ci u(N_\alpha(t) - N_\alpha(s))\} \vert \mathcal{F}_s] \nonumber\\
&= \mathbb{E}\left[\exp\{\ci u(N_\alpha(t) - N_\alpha(s))\} \vert \mathcal{F}_0 \vee \mathcal{F}_s^{N_\alpha}\right]\nonumber\\
&= \mathbb{E}\left[\exp\{\ci u(N_1(\Lambda(Y_\alpha(t))) - N_1(\Lambda(Y_\alpha(s))))\}\vert \mathcal{F}_0 \vee \mathcal{F}_{\Lambda(Y_\alpha(s))}^{N_1}\right] \label{eq:25}\\
&= \mathbb{E}\left[\exp\{\ci u(N_1(\Lambda(Y_\alpha(t))) - N_1(\Lambda(Y_\alpha(s))))\}\vert \mathcal{F}_0\right] \label{eq:26}\\
     &= \exp[\Lambda(Y_\alpha(s), Y_\alpha(t))(e^{\ci u} - 1)],\nonumber
   \end{align}
   where in (\ref{eq:25}) we used the time-change theorem (see for example Thm. 7.4.I. p. 258 in \cite{Daley_2003}) and in (\ref{eq:26}) the fact that the standard Poisson process has independent increments. 
This means, conditional on $(\mathcal{F}_t)_{t\geq 0}$, $(N_\alpha(t))$ has independent increments and 
\begin{equation*}
  (N_\alpha(t) - N_\alpha(s)) \vert \mathcal{F}_s \sim \text{Poi}(\Lambda(Y_\alpha(s), Y_\alpha(t)))\stackrel{d}{=}\text{Poi}(\Lambda(Y_\alpha(t)) - \Lambda(Y_\alpha(s))).
\end{equation*}
Thus, $(N(Y_\alpha(t)))$ is a Cox process directed by $\Lambda(Y_\alpha(t))$ by definition. 
\end{proof}

\section{The FNPP and its compensator}
\label{sec:fract-poiss-proc}
The idenfication of the FNPP as a Cox process in the previous section allows us to determine the compensator of the FNPP. In fact, the compensator of a Cox process coincides with its directing process. From Lemma 6.6.3. p.194 in \cite{Bielecki_2002} we have the result
\begin{proposition}
\label{thm:fpp-its-compensator}
  Let the FNPP be adapted to the filtration $(\mathcal{F}_t)$ as in (\ref{eq:20}) with non-trivial initial history $\mathcal{F}_0 := \sigma(\{Y_\alpha(t), t\geq 0\})$. Assume $\mathbb{E}[\Lambda(Y_\alpha(t))] < \infty$. Then the FNPP has $\mathcal{F}_t$-compensator $(A(t))_{t\geq 0}$, where $A(t) := \Lambda(Y_\alpha(t))$, i.e. the stochastic process $(M(t))_{t\geq 0}$ defined by $M(t) := N(Y_\alpha(t)) - \Lambda(Y_\alpha(t))$ is a $\mathcal{F}_t$-martingale.
\end{proposition}

\subsection{A central limit theorem}
\label{sec:centr-limit-theor-1}
Using the compensator of the FNPP, we can apply martingale methods in order to derive limit theorems for the FNPP. For the sake of completeness, we restate the definition of $\mathcal{F}_0$-stable convergence along with the theorem which will be used later.

\begin{definition}
  If $(X_n)_{n \in \mathbb{N}}$ and X are $\mathbb{R}$-valued random variables on a probability space $(\Omega, \mathcal{E}, \mathbb{P})$ and $\mathcal{F}$ is a sub-$\sigma$-algebra of $\mathcal{E}$, then $X_n \to X$ ($\mathcal{F}$-stably) in distribution if for all $B\in \mathcal{F}$ and all $A\in \mathcal{B}(\mathbb{R})$ with $ \mathbb{P}(X \in \partial A) = 0$,
  \begin{equation*}
    \mathbb{P}(\{X_n \in A\} \cap B) \xrightarrow[n \to \infty]{} \mathbb{P}(\{X \in A\} \cap B)
  \end{equation*}
(see Definition A.3.2.III. in \cite{Daley_2003}). 
\end{definition}
Note that $\mathcal{F}$-stable converges implies weak convergence/convergence in distribution. We can derive a central limit theorem for the FNPP using Corollary 14.5.III. in \cite{Daley_2003} which we state here as a lemma for convenience.
\begin{lemma}
\label{thm:centr-limit-theor-2}
  Let $N$ be a simple point process on $\mathbb{R}_{+}$, $(\mathcal{F}_t)_{t\geq 0}$-adapted and with continuous $(\mathcal{F}_t)_{t\geq 0}$-compensator $A$. Set
  \begin{equation*}
    X_T := \int_0^T f_T(u) [\sd N(u) - \sd A(u)].
  \end{equation*}
Suppose for each $T > 0$ an $(\mathcal{F}_t)_{t\geq 0}$-predictable process $f_T(t)$ is given such that 
  \begin{equation*}
    B_T^2 = \int_0^T [f_T(u)]^2 \sd A(u) > 0.
  \end{equation*}
Then the randomly normed integrals $X_T/B_T$ converge $\mathcal{F}_0$-stably to a standard normal variable $W \sim N(0,1)$.
\end{lemma}

Note that the above integrals are well-defined as explained in Point 2 in Remarks \ref{rem:fnpp-as-cox}. The above theorem allows us to show the following result for the FNPP.
\begin{proposition}
  Let $(N(Y_\alpha(t)))_{t\geq 0}$ be the FNPP adapted to the filtration $(\mathcal{F}_t)_{t\geq 0}$ as defined in Section \ref{sec:fract-comp-as}. Then,
  \begin{equation}
    \label{eq:7}
    \frac{N(Y_\alpha(T)) - \Lambda(Y_\alpha(T))}{\sqrt{\Lambda(Y_\alpha(T))}} \xrightarrow[T \to \infty]{} W \sim N(0,1) \qquad \mathcal{F}_0\text{-stably}.
  \end{equation}
\end{proposition}
\begin{proof}
  First note that the compensator $A(t) := \Lambda(Y_\alpha(t))$ is continuous in $t$.
  Let $f_T(u) \equiv C_T $ a constant, then
  \begin{align*}
    B_T^2 &= \int_0^T [f_T(u)]^2 \sd A(u) = C_T^2 (\Lambda(Y_\alpha(T)) - \Lambda(Y_\alpha(0)))\\
    &= C_T^2 \Lambda(Y_\alpha(T)) > 0, \: \forall T > 0
  \end{align*}
and
\begin{align*}
  X_T &:= \int_0^T f_T(u) [\sd N(Y_\alpha(u)) - \sd A(u)]\\
      &= C_T [N(Y_\alpha(T)) - A(T) - N(Y_\alpha(0)) + A(0)]\\
  &= C_T [N(Y_\alpha(T)) - A(T)] = C_T [N(Y_\alpha(T)) - \Lambda(Y_\alpha(T))].
\end{align*}
It follows from Theorem \ref{thm:centr-limit-theor-2} above that
\begin{align*}
  \frac{X_T}{B_T} &= \frac{C_T[N(Y_\alpha(T)) - \Lambda(Y_\alpha(T))]}{\sqrt{C_T^2 \Lambda(Y_\alpha(T))}} \\
  &= \frac{N(Y_\alpha(T)) - \Lambda(Y_\alpha(T))}{\sqrt{\Lambda(Y_\alpha(T))}} \xrightarrow[T\to \infty]{} W \sim N(0,1) \qquad \mathcal{F}_0\text{-stably}.
\end{align*}
\end{proof}

\subsection{Limit $\alpha \rightarrow 1$}
In the following, we give a more rigorous proof for the limit $\alpha \rightarrow 1$ in Section 3.2(ii) in \cite{Leonenko_2017}.

\begin{proposition}
\label{prop:limit-alpha-right}
  Let the FNPP be adapted to the filtration $(\mathcal{F}_t)$ as in (\ref{eq:20}) with non-trivial initial history $\mathcal{F}_0 := \sigma(\{Y_\alpha(t), t\geq 0\})$. Let $(N_\alpha(t))_{t\geq 0}$ be the FNPP as defined in (\ref{eq:23}). Then, we have the limit
  \begin{equation*}
    N_\alpha \xrightarrow[\alpha \to 1]{J_1} N \quad \text{ in } \quad D([0, \infty)).
  \end{equation*}
\end{proposition}
\begin{proof}
By Proposition \ref{thm:fpp-its-compensator} we see that $(\Lambda(Y_\alpha(t)))_{t \geq 0}$ is the compensator of $(N_\alpha(t))_{t\geq 0}$.
According to Theorem VIII.3.36 on p. 479 in \cite{Jacod_Shiryaev_2003} if suffices to show
\begin{equation*}
  \Lambda(Y_\alpha(t)) \xrightarrow[\alpha \rightarrow 1]{\mathcal{P}}\Lambda(t).
\end{equation*}
We can check that the Laplace transform of the density of the inverse $\alpha$-stable subordinator converges to the Laplace transform of the delta distribution:
\begin{equation}
  \label{eq:3}
  \mathcal{L}\{h_\alpha(\cdot,y)\}(s,y) = E_\alpha(-ys^\alpha) \xrightarrow[]{\alpha \rightarrow 1} \e^{-ys} = \mathcal{L}\{\delta_0(\cdot-y)\}(s,y).
\end{equation}
We may take the limit as the power series representation of the (entire) Mittag-Leffler function is absolutely convergent. Thus (\ref{eq:3}) implies
\begin{equation*}
  Y_\alpha(t) \xrightarrow[\alpha \rightarrow 1]{d} t \qquad \forall t \in \mathbb{R}_{+}.
\end{equation*}
As convergence in distribution to a constant automatically improves to convergence in probability, we have
\begin{equation*}
  Y_\alpha(t) \xrightarrow[\alpha \rightarrow 1]{\mathcal{P}} t \qquad \forall t \in \mathbb{R}_{+}.
\end{equation*}
By the continuous mapping theorem, it follows that
\begin{equation*}
  \Lambda(Y_\alpha(t)) \xrightarrow[\alpha \rightarrow 1]{\mathcal{P}} \Lambda(t) \qquad \forall t \in \mathbb{R}_{+},
\end{equation*}
which concludes the proof.
\end{proof}

\section{Regular variation and scaling limits}
\label{sec:scaling-limit}
In this section we will work with the trivial initial filtration setting ($\mathcal{F}_0 = \{\varnothing, \Omega\}$), i.e. $\mathcal{F}_t$ is assumed to be the natural filtration of the FNPP. In this setting, the FNPP can generally not be seen as a Cox process and although the compensator of the FNPP does exist, it is difficult to give a closed form expression for it.\\
Instead, we follow the approach of results given in \cite{Grandell_1976}, \cite{Serfozo_1972a}, \cite{Serfozo_1972b}, which require conditions on the function $\Lambda$. Recall that a function $\Lambda$ is \textit{regularly varying with index} $\beta\in \mathbb{R}$ if 
\begin{equation}
    \label{eq:2}
    \frac{\Lambda(xt)}{\Lambda(t)} \xrightarrow[t \rightarrow \infty]{} x^\beta, \qquad \forall x > 0.
\end{equation}
Under the mild condition of measurability, one can show that the above limit is quite general in the sense that if the quotient of the right hand side of (\ref{eq:2}) converges to a function $x \mapsto g(x)$, $g$ has to be of the form $x^\beta$ (see Thm. 1.4.1 in \cite{Bingham_1989}).
\begin{example}
  We check whether typical rate functions (taken from Remark 2 in \cite{Leonenko_2017}) fulfill the regular variation condition.
  \begin{itemize}
  \item[(i)] Weibull's rate function 
    \begin{equation*}
  \Lambda(t)=\left( \frac{t}{b}\right)^c, \quad \lambda(t) = \frac{c}{b}\left( \frac{t}{b}\right)^{c-1},\quad c\geq 0,\, b > 0
\end{equation*}
is regulary varying with index $c$. This can be seen as follows
\begin{equation*}
  \frac{\Lambda(xt)}{\Lambda(t)} = \frac{(xt)^c}{t^c} = x^c, \quad \forall x > 0.
\end{equation*}
\item[(ii)] Makeham's rate function
\begin{equation*}
  \Lambda(t) = \frac{c}{b}\e^{bt}-\frac{c}{b} + \mu t, \quad \lambda(t) = c\e^{bt}+\mu, \quad c>0,\, b>0,\, \mu\geq 0
\end{equation*}
is not regulary varying, since
\begin{align*}
  \frac{\Lambda(xt)}{\Lambda(t)} &= \frac{(c/b)\e^{bxt} - (c/b) + \mu xt}{(c/b) \e^{bt} - (c/b) + \mu t}
  = \frac{(c/b)\e^{bt(x-1)} - (c/b)\e^{-bt} + \mu xt\e^{-bt}}{(c/b) - (c/b)\e^{-bt} + \mu t\e^{-bt}}\\
  &\xrightarrow[]{t\to \infty}\left\{ 
    \begin{array}{ll}
      0 & \text{ if } x < 1\\
      1 & \text{ if } x=1\\
      +\infty & \text{ if } x > 1
    \end{array}
\right.
\end{align*}
does not fulfill (\ref{eq:2}). \triag
\end{itemize}
\end{example}
In the following, the condition that $\Lambda$ is regularly varying is useful for proving limit results. We will first show a one-dimensional limit theorem before moving on to the functional analogue.

\subsection{A one-dimensional limit theorem}
\label{sec:one-dimens-limit}

\begin{theorem}
  \label{thm:regul-vari-scal}
  Let the FNPP $(N_\alpha(t))_{t\geq 0}$ be defined as in Equation (\ref{eq:23}). Suppose the function $t\mapsto \Lambda(t)$ is regularly varying with index $\beta \in \mathbb{R}$. Then the following limit holds for the FNPP:
\begin{equation}
  \label{eq:4}
  \frac{N_\alpha(t)}{\Lambda(t^\alpha)} \xrightarrow[t \rightarrow \infty]{d} (Y_\alpha(1))^\beta.
\end{equation}
\end{theorem}
\begin{proof}
  We will first show that the characteristic function of the random variable on the left hand side of (\ref{eq:4}) converges to the characteristic function of the right hand side.\\
  By self-similarity of $Y_\alpha$ we have 
  \begin{equation*}
    N_1(\Lambda(Y_\alpha(t))) \stackrel{d}{=} N_1(\Lambda(t^\alpha Y_\alpha(1))). 
  \end{equation*}
Therefore, it follows for the characteristic function of $Z(t) := \frac{N_\alpha(t)}{\Lambda(t^\alpha)}$ that 
\begin{align}
  \varphi(t) &:= \mathbb{E}[\exp(\ci u Z(t))] 
  = \mathbb{E}[\exp(\ci u \Lambda(t^\alpha)^{-1} N_1(\Lambda(Y_\alpha(t))))] \nonumber\\
  &= \mathbb{E}[\exp(\ci u \Lambda(t^\alpha)^{-1} N_1(\Lambda(t^\alpha Y_\alpha(1))))]\nonumber\\
  &= \int_0^\infty \mathbb{E}[\exp(\ci u \Lambda(t^\alpha)^{-1} N_1(\Lambda(t^\alpha x)))] h_\alpha(1,x) \,\sd x\label{eq:32}\\
  &= \int_0^\infty \exp(\Lambda(t^\alpha x)(\e^{\ci u\Lambda(t^\alpha)^{-1}} - 1)) h_\alpha(1,x) \,\sd x \label{eq:33},
\end{align}
where we used a conditioning argument in (\ref{eq:32}), $x \mapsto h_\alpha(1,x)$ is the density function of the distribution of $Y_\alpha(1)$. In the last step in (\ref{eq:33}) we may insert the characteristic function of a Poisson distributed random variable with parameter $\Lambda(t^\alpha x)$ evaluated at the point $u\Lambda(t^\alpha)^{-1}$.\\
In order to pass to the limit, we need to justify that we may exchange integration and limit. It can be observed that the integrand is dominated by an integrable function independent of $t$. By Jensen's inequality
\begin{align*}
  \left\vert\mathbb{E}[\exp(\ci u \Lambda(t^\alpha)^{-1} \right.&\left.N_1(\Lambda(t^\alpha x)))] h_\alpha(1,x)\right\vert\\ 
&\leq \mathbb{E}[\vert\exp(\ci u \Lambda(t^\alpha)^{-1} N_1(\Lambda(t^\alpha x)))\vert] h_\alpha(1,x)
  \leq h_\alpha(1,x)
\end{align*}
This allows us to use the dominated convergence theorem to get
\begin{align}
  \lim_{t\rightarrow \infty} \varphi(t) &= \lim_{t\rightarrow\infty} \int_0^\infty \exp(\Lambda(t^\alpha x)(\e^{\ci u\Lambda(t^\alpha)^{-1}} - 1)) h_\alpha(1,x) \,\sd x \nonumber\\
  &=  \int_0^\infty \left[\lim_{t\rightarrow\infty} \exp(\Lambda(t^\alpha x)(\e^{\ci u\Lambda(t^\alpha)^{-1}} - 1)) \right]h_\alpha(1,x) \,\sd x. \label{eq:35}
\end{align}
We are left with calculating the limit in the square bracket in (\ref{eq:35}). To this end, consider a power series expansion of $\e^{\ci u\Lambda(t^\alpha)^{-1}}$ to observe that
\begin{align*}
  \exp\left( \Lambda(t^\alpha x)(\e^{\ci u\Lambda(t^\alpha)^{-1}} - 1)\right) &= \exp\left(\Lambda(t^\alpha x) \left(\sum_{k=1}^\infty \frac{(\ci u)^k}{\Lambda(t^\alpha)^k k!} \right)\right)\\
&= \exp\Bigg( \frac{\ci u}{1!} \underbrace{\frac{\Lambda(t^\alpha x)}{\Lambda(t^\alpha)}}_{\xrightarrow[t\rightarrow \infty]{} x^\beta} + \underbrace{\Lambda(t^\alpha x) \mathcal{O}\left(\frac{1}{\Lambda(t^\alpha)^2}\right)}_{\xrightarrow[t\rightarrow \infty]{} 0} \Bigg),
\end{align*}
where we have used that $\Lambda$ is regularly varying with index $\beta$ in the last step. Inserting this result into (\ref{eq:35}) yields
\begin{align*}
  \lim_{t\rightarrow \infty} \varphi(t) &= \int_0^\infty \exp\left( \ci u x^\beta\right) h_\alpha(1,x) dx\\
  &= \mathbb{E}[\e^{\ci u (Y_\alpha(1))^\beta}].
\end{align*}
Applying Lévy's continuity theorem concludes the proof.
\end{proof}
\begin{remark}
The above result can be shown alternatively using Theorem 3.4 in \cite{Serfozo_1972a} or Theorem 1 on pp. 69-70 in \cite{Grandell_1976}. The limit distribution of $N_\alpha(t)/\Lambda(t^\alpha)$ is the sum of the limit distribution $(Y_\alpha(1))^\beta$ of the inner process $\Lambda(Y_\alpha(t))$ and a normal distribution (the limit of the outer process, the Poisson process). The variance of the normal distribution is determined by the norming constants in the inner process limit. In our case the variance is $0$ and we are left with $(Y_\alpha(1))^\beta$ as limit of the overall process.
\end{remark}
\begin{remark}
  As a special case of the theorem we get for $\Lambda(t) = \lambda t$, for constant $\lambda > 0$ 
  \begin{equation*}
    \frac{\Lambda(xt)}{\Lambda(t)} = x^1
  \end{equation*}
which means $\Lambda$ is regularly varying with index $\beta=1$. It follows that
\begin{equation*}
  \frac{N_1(\lambda Y_\alpha(t))}{\lambda t^\alpha} \xrightarrow[t\rightarrow\infty]{d} Y_\alpha(1).
\end{equation*}
This is in accordance to the scaling limit given in \cite{Cahoy_2010} who showed
\begin{equation*}
  \frac{N_1(\lambda Y_\alpha(t))}{\mathbb{E}[N_1(\lambda Y_\alpha(t))]} 
= \frac{N_1(\lambda Y_\alpha(t))}{\frac{\lambda t^\alpha}{\Gamma(1+\alpha)}} 
\xrightarrow[t\rightarrow\infty]{d} \Gamma(1+\alpha) Y_\alpha(1).
\end{equation*}
\end{remark}

\subsection{A functional limit theorem}
\label{sec:funct-limit-theor}

The one-dimensional result in Theorem \ref{thm:regul-vari-scal} can be extended to a functional limit theorem. In the following we consider the Skorohod space $\mathcal{D}([0,\infty))$ endowed with a suitable topology (we will focus on the $J_1$ and $M_1$ topology). For more details see \cite{Meerschaert_Sikorskii_2012}.
\begin{theorem}
\label{thm:regul-vari-fct}
  Let the FNPP $(N_\alpha(t))_{t\geq 0}$ be defined as in Equation (\ref{eq:23}). Suppose the function $t\mapsto \Lambda(t)$ is regularly varying with index $\beta \in \mathbb{R}$. Then the following limit holds for the FNPP:
\begin{equation}
  \label{eq:29}
   \left(\frac{N_\alpha(t\tau)}{\Lambda(t^\alpha)}\right)_{\tau\geq 0} \xrightarrow[t \rightarrow \infty]{J_1} \left([Y_\alpha(\tau)]^\beta\right)_{\tau\geq 0}.
\end{equation}
\end{theorem}
\begin{remark}
  As the limit process has continuous paths the mode of convergence improves to local uniform convergence. Also in this theorem, we will denote the homogeneous Poisson process with intensity parameter $\lambda = 1$ with $N_1$.
\end{remark}
In order to proof the Theorem we need Theorem 2 on p. 81 in \cite{Grandell_1976}, which we will state here for convenience.
\begin{theorem}
  \label{thm:regul-vari-fct-grandell}
  Let $\bar{\Lambda}$ be a stochastic process in $\mathcal{D}([0,\infty))$ with $\bar{\Lambda}(0) = 0$ and let $N = N_1(\bar{\Lambda})$ be the corresponding doubly stochastic process. Let $a \in \mathcal{D}([0,\infty))$ with $a(0) = 0$ and $t\mapsto b_t$ a positive regularly varying function with index $\rho > 0$ such that
  \begin{gather*}
    \frac{a(t)}{b_t} \xrightarrow[t \to \infty]{} \kappa \in [0,\infty) \text{ and }\\
    \left( \frac{\bar{\Lambda}(t\tau) - a(t\tau)}{b_t}\right)_{\tau \geq 0} \xrightarrow[t \to \infty]{J_1} (S(\tau))_{\tau \geq 0},
  \end{gather*}
where $S$ is a stochastic process in $\mathcal{D}([0,\infty))$. Then
\begin{equation*}
  \left( \frac{N(t\tau) - a(t\tau)}{b_t}\right)_{\tau \geq 0} \xrightarrow[t\to \infty]{J_1} (S(\tau) +  h(B(\tau)))_{\tau \geq 0},
\end{equation*}
where $h(\tau) = \kappa \tau^{2\rho}$ and $(S(t))_{t \geq 0}$ and $(B(t))_{t \geq 0}$ are independent. $(B(t))_{t \geq 0}$ is the standard Brownian motion in $\mathcal{D}([0,\infty))$.
\end{theorem}
\begin{proof}[Proof of Thm. \ref{thm:regul-vari-fct}]
  We apply Theorem \ref{thm:regul-vari-fct-grandell} and choose $a\equiv 0$ and $b_t=\Lambda(t^\alpha)$. Then it follows that $\kappa = 0$ and it can be checked that $b_t$ is regularly varying with index $\alpha\beta$:
  \begin{equation*}
    \frac{b_{xt}}{b_t} = \frac{\Lambda(x^\alpha t^\alpha)}{\Lambda(t^\alpha)} \xrightarrow[t\to \infty]{} x^{\alpha\beta}
  \end{equation*}
by the regular variation property in (\ref{eq:2}).\\
We are left to show that
\begin{equation}
  \label{eq:6}
  \tilde{\Lambda}_t(\tau) := \left(\frac{\Lambda(Y_\alpha(t\tau))}{\Lambda(t^\alpha)}\right)_{\tau \geq 0} \xrightarrow[t\to \infty]{J_1} \left([Y_\alpha(\tau)]^\beta\right)_{\tau \geq 0}.
\end{equation}
This can be done by following the usual technique of first proving convergence of the finite-dimensional marginals and then tightness of the sequence in the Skorohod space $\mathcal{D}([0,\infty))$.\\
Concerning the convergence of the finite-dimensional marginals we show convergence of their respective characteristic functions. Let $t>0$ be fixed at first, $\tau = (\tau_1, \tau_2, \ldots, \tau_n) \in \mathbb{R}_+^n$ and $\langle \cdot, \cdot \rangle$ denote the scalar product in $\mathbb{R}^n$. Then, we can write the characteristic function of the joint distribution of the vector
\begin{equation*}
  \frac{\Lambda(t^\alpha Y_\alpha(\tau))}{\Lambda(t^\alpha)} = \left( \frac{\Lambda(t^\alpha Y_\alpha(\tau_1))}{\Lambda(t^\alpha)}, \frac{\Lambda(t^\alpha Y_\alpha(\tau_2))}{\Lambda(t^\alpha)}, \ldots, \frac{\Lambda(t^\alpha Y_\alpha(\tau_n))}{\Lambda(t^\alpha)}\right) \in \mathbb{R}_{+}^n
\end{equation*}
 as
\begin{align*}
  \varphi_t(u) &:= \mathbb{E}\left[\exp\left(\ci \left\langle u , \frac{\Lambda(Y_\alpha(t\tau))}{\Lambda(t^\alpha)} \right\rangle\right)\right]
  = \mathbb{E}\left[\exp\left(\ci \left\langle u , \frac{\Lambda(t^\alpha Y_\alpha(\tau))}{\Lambda(t^\alpha)} \right\rangle\right)\right]\\
  &= \int_{\mathbb{R}_{+}^n} \exp\left(\ci \left\langle u,  \frac{\Lambda(t^\alpha x)}{\Lambda(t^\alpha)}\right\rangle\right)h_\alpha(\tau, x) \sd x\\
  &= \int_{\mathbb{R}_{+}^n} \left[\prod_{k=1}^n \exp\left(\ci u_k \frac{\Lambda(t^\alpha x_k)}{\Lambda(t^\alpha)}\right)\right]h_\alpha(\tau_1, \ldots, \tau_n; x_1, \ldots, x_n) \sd x_1 \ldots \sd x_n
\end{align*}
where $u \in \mathbb{R}^n$ and  $h_\alpha(\tau, x) = h_\alpha(\tau_1, \tau_2, \ldots, \tau_n; x_1, x_2 \ldots, x_n)$ is the density of the joint distribution of $(Y_\alpha(\tau_1), Y_\alpha(\tau_2), \ldots, Y_\alpha(\tau_n))$. We can find a dominating function by the following estimate:
\begin{equation*}
  \left\vert \exp\left(\ci \left\langle u,  \frac{\Lambda(t^\alpha x)}{\Lambda(t^\alpha)}\right\rangle\right)h_\alpha(\tau, x)\right\vert \leq h_\alpha(\tau, x).
\end{equation*}
The upper bound is an integrable function which is independent of $t$. By dominated convergence we may interchange limit and integration:
\begin{align*}
  \lim_{t\to \infty} \varphi_n(u) &= \lim_{t \to \infty} \int_{\mathbb{R}_{+}^n} \exp\left(\ci \left\langle u,  \frac{\Lambda(t^\alpha x)}{\Lambda(t^\alpha)}\right\rangle\right)h_\alpha(\tau, x) \sd x\\
  &= \int_{\mathbb{R}_{+}^n} \lim_{t\to \infty}\exp\left(\ci \left\langle u,  \frac{\Lambda(t^\alpha x)}{\Lambda(t^\alpha)}\right\rangle\right)h_\alpha(\tau, x) \sd x\\
  &= \int_{\mathbb{R}_{+}^n} \exp\left(\ci \left\langle u, x^\beta\right\rangle\right)h_\alpha(\tau, x) \sd x = \mathbb{E}[\exp(i\langle u, (Y_\alpha(\tau))^\beta\rangle)],
\end{align*}
where in the last step we used the continuity of the exponential function and the scalar product to calculate the limit. By L\'evy's continuity theorem we may conclude that for $n\in \mathbb{N}$
\begin{equation*}
  \left(\frac{\Lambda(Y_\alpha(t\tau_k))}{\Lambda(t^\alpha)}\right)_{k=1, \ldots, n} \xrightarrow[t \to \infty]{d} \left([Y_\alpha(\tau_k)]^\beta\right)_{k=1, \ldots, n}.
\end{equation*}
In order to show tightness, first observe that for fixed $t$ both the stochastic process $\tilde{\Lambda}_t$ on the left hand side and the limit candidate $([Y_\alpha(\tau)]^\beta)_{\tau \geq 0}$ have increasing paths. Moreover, the limit candidate has continuous paths. Therefore we are able to invoke Thm. VI.3.37(a) in \cite{Jacod_Shiryaev_2003} to ensure tightness of the sequence $(\tilde{\Lambda}_t)_{t\geq 0}$ and thus the assertion follows.
\end{proof}

By applying the transformation theorem for probability densities to (\ref{eq:5}), we can write for the density $h_\alpha^\beta(t, \cdot)$ of the one-dimensional marginal of the limit process $([Y_\alpha(t)]^\beta)_{t\geq 0}$ as
\begin{align}
  h_\alpha^\beta(t,x) &= \frac{1}{\beta} x^{1/\beta - 1} h_\alpha(t,x^{1/\beta})\nonumber\\
  &= \frac{1}{\beta} x^{1/\beta - 1} \frac{t}{\alpha x^{1/\beta(1+1/\alpha)}} g_\alpha\left(\frac{t}{y^{1/(\alpha\beta)}}\right)\nonumber\\
  &= \frac{t}{\alpha\beta x^{1+1/(\alpha\beta)}}g_\alpha\left(\frac{t}{y^{1/(\alpha\beta)}}\right).\label{eq:10}
\end{align}
Note that this is \textit{not} the density of $Y_{\alpha\beta}(t)$.


A further limit result can be obtained for the FHPP via a continuous mapping argument.
\begin{proposition}
Let $(N_1(t))_{t \geq 0}$ be a homogeneous Poisson process and $(Y_\alpha(t))_{t\geq 0}$ be the inverse $\alpha$-stable subordinator. Then
  \begin{equation*}
      \left(\frac{N_1(Y_\alpha(t)) - \lambda Y_\alpha(t)}{\sqrt{\lambda}}\right)_{t \geq 0} \xrightarrow[\lambda \to \infty]{J_1} (B(Y_\alpha(t)))_{t \geq 0},
    \end{equation*}
where $(B(t))_{t \geq 0}$ is a standard Brownian motion.
\end{proposition}
\begin{proof}
The classical result
    \begin{equation*}
      \left(\frac{N_1(t) - \lambda t}{\sqrt{\lambda}}\right)_{t \geq 0} \xrightarrow[\lambda \to \infty]{J_1} (B(t))_{t \geq 0}
    \end{equation*}
can be shown by using that $(N(t) - \lambda t)_{t\geq 0}$ is a martingale. As $(B(t))_{t \geq 0}$ has continuous paths and $(Y_\alpha(t))_{t\geq 0}$ has increasing paths we may use Theorem 13.2.2 in \cite{Whitt_2002} to obtain the result.
\end{proof}
The above proposition can be compared with Lemma \ref{lem:one-dimens-limit-2} in the next section and a similar continuous mapping argument is applied in the proof of Theorem \ref{thm:funct-limit-theor-1}.

\section{The fractional compound Poisson process}
\label{sec:appl-fract-comp}

Let $X_1, X_2, \ldots$ be a sequence of i.i.d. random variables. The fractional compound Poisson process is defined analogouly to the standard Poisson process where the Poisson process is replaced by a fractional one:
\begin{equation*}
  Z_\alpha(t) := \sum_{k=1}^{N_\alpha(t)} X_k,
\end{equation*}
where $\sum_{k=1}^0 X_k := 0$. The process $N_\alpha$ is not necessarily independent of the $X_i$'s unless stated otherwise.\\
Stable laws can be defined via the form of their characteristic function.
\begin{definition}
  A random variable $S$ is said to have stable distribution if there are parameters $0 < \alpha \leq 2$, $\sigma \geq 0$, $-1 \leq \beta \leq 1$ and $\mu \in \mathbb{R}$ such that its characteristic function has the following form:
  \begin{equation*}
    \mathbb{E}[\exp(i\theta S)] = \left\{
      \begin{array}{ll}
        \exp\left( -\sigma^\alpha \vert \theta\vert^\alpha \left[ 1 - \ci \beta\, \text{sign}(\theta) \tan\left( \frac{\pi \alpha}{2}\right)\right] + \ci \mu\theta\right)& \text{ if } \alpha \neq 1,\\
        \exp\left( -\sigma\vert \theta\vert \left[ 1 + \ci \beta \frac{2}{\pi}\, \text{sign}(\theta)\ln(\vert \theta\vert)\right] + \ci \mu\theta\right) & \text{ if } \alpha = 1
      \end{array}
      \right.
  \end{equation*}
(see Definition 1.1.6 in \cite{Samorodnitsky_Taqqu_1994}).
We will assume a limit result for the sequence of partial sums without time-change
\begin{equation*}
  S_n := \sum_{k=1}^n X_k,
\end{equation*}
usually a stable limit, i.e. there exist sequences $(a_n)_{n\in \mathbb{N}}$ and $(b_n)_{n\in \mathbb{N}}$ and and a random variable following a stable distribution $S$ such that
\begin{equation*}
  \bar{S}_n := a_n S_n - b_n \xrightarrow[n \to \infty]{d} S.
\end{equation*}
(for details see for example Chapter XVII in \cite{Feller_1971}). In other words the distribution of the $X_k$'s is in the domain of attraction of a stable law.
\end{definition}
In the following we will derive limit theorems for the fractional compound Poisson process. In Section \ref{sec:funct-limit-theor-2}, we assume $N_\alpha$ to be independent of the $X_k$'s and use a continuous mapping theorem argument to show functional convergence w.r.t. a suitable Skorohod topology. A corresponding one-dimensional limit theorem would follow directly from the functional one. However, in the special case of $N_\alpha$ being a FHPP, using Anscombe type theorems in Section 6.1 allows us to drop the independence assumption between $N_\alpha$ and the $X_k$'s and thus strengthen the result for the one-dimensional limit.
\subsection{A one-dimensional limit result}
\label{sec:one-dimens-limit-1}

The following theorem is due to \cite{Anscombe_1952} and can be found slightly reformulated in \cite{Richter_1965}.
\begin{theorem}
  We assume that the following conditions are fulfilled:
  \begin{itemize}
  \item[(i)] The sequence of random variables $R_n$ such that
    \begin{equation*}
      R_n \xrightarrow[n\to \infty]{d} R,
    \end{equation*}
    for some random variable $R$.
  \item[(ii)] Let the family of integer-valued random variables $N(t)$ be relatively stable, i.e. for a real-valued function $\psi$ with $\psi(t) \xrightarrow[t\to \infty]{<} \infty$ it holds that
    \begin{equation*}
      \frac{N(t)}{\psi(t)} \xrightarrow[t \to \infty]{P} 1.
    \end{equation*}
  \item[(iii)] (Uniform continuity in probability) For every $\varepsilon > 0$ and $\eta > 0$ there exists a $c = c(\varepsilon, \eta)$ and a $t_0 = t_0(\varepsilon, \eta)$ such that for all $t \geq t_0$
    \begin{equation*}
      \mathbb{P}\left( \max_{m: \vert m-t \vert < ct} \vert R_m - R_t\vert > \varepsilon \right) < \eta.
    \end{equation*}
  \end{itemize}
Then,
\begin{equation*}
  R_{N(t)} \xrightarrow[t\to \infty]{d} R.
\end{equation*}
\end{theorem}

We would like to use the above theorem for $R_n = S_n$. Indeed, condition (i) follows from the assumption that the law of $X_1$ lies in the domain of attraction of a stable law. It is readily verified in Theorem 3 in \cite{Anscombe_1952} that $(S_n)$ fulfills the condition (iii). Concerning the condition (ii), note that the required convergence in probability is stronger than the convergence in distribution we have derived in the previous sections for the FNPP. Nevertheless, in the special case of the FHPP, we can improve the mode of convergence.
\begin{lemma}
  \label{lem:one-dimens-limit-2}
  Let $N_\alpha$ be a FHPP, i.e. $\Lambda(t) = \lambda t$ in (\ref{eq:23}). Then with $C := \frac{\lambda}{\Gamma(1+\alpha)}$ it holds that
  \begin{equation*}
    \frac{N_\alpha(t)}{C t^\alpha} \xrightarrow[t \to \infty]{P} 1.
  \end{equation*}
\end{lemma}
\begin{proof}
  According to Proposition 4.1 from \cite{DiCrescenzo_2016} we have the result that for fixed $t > 0$ the convergence
  \begin{equation}
    \label{eq:9}
    \frac{N_1(\lambda Y_\alpha(t))}{\mathbb{E}[N_1(\lambda Y_\alpha(t))]} = \frac{N_1(\lambda Y_\alpha(t))}{\frac{\lambda t^\alpha}{\Gamma(1+\alpha)}} \xrightarrow[\lambda \to \infty]{L^1} 1
  \end{equation}
holds and therefore also in probability.\\
It can be shown by using the fact that the moments and the waiting time distribution of the FHPP can be expressed in terms of the Mittag-Leffler function.\\
Let $\varepsilon >0$. We have
\begin{align}
  \lim_{t\to \infty} \mathbb{P}\left(\left\vert \frac{N_1(\lambda Y_\alpha(t))}{C t^\alpha} - 1\right\vert > \varepsilon\right)
  &= \lim_{t\to \infty} \mathbb{P}\left(\left\vert \frac{N_1(\lambda t^\alpha Y(1))}{\frac{\lambda t^\alpha}{\Gamma(1+\alpha)}} - 1\right\vert > \varepsilon\right) \label{eq:14}\\
&= \lim_{\tau\to \infty} \mathbb{P}\left(\left\vert \frac{N_1(\tau Y(1))}{\frac{\tau \cdot 1^\alpha}{\Gamma(1+\alpha)}} - 1\right\vert > \varepsilon\right) = 0,\label{eq:15}
\end{align}
where in (\ref{eq:14}) we used the self-similarity property of $Y_\alpha$ and in (\ref{eq:15}) we applied (\ref{eq:9}) with $t = 1$.
\end{proof}
By applying Lemma \ref{lem:one-dimens-limit-2} condition (ii) is satisfied and it follows that
\begin{equation*}
  \bar{S}_{N_\alpha(t)} = a_{N_\alpha(t)}\sum_{k=1}^{N_\alpha(t)} X_k - b_{N_\alpha(t)} \xrightarrow[t\to \infty]{d} S
\end{equation*} 
and (see Theorem 3.6 in \cite{Gut_2013})
\begin{equation*}
  \bar{S}_{N_\alpha(t)} = a_{\lfloor C t \rfloor}\sum_{k=1}^{N_\alpha(t)} X_k - b_{\lfloor C t \rfloor} = a_{\lfloor C t \rfloor} Z_\alpha(t) - b_{\lfloor C t \rfloor} \xrightarrow[t\to \infty]{d} S.
\end{equation*}
Note that this convergence result does not require $N_\alpha$ to be independent of the $X_k$'s. The above derivation also works for mixing sequences $X_1, X_2, \ldots$ instead of i.i.d. (see \cite{Csorgo_Fischler_1973} for a generalisation of Anscombe's theorem for mixing sequences).


\subsection{A functional limit theorem}
\label{sec:funct-limit-theor-2}

\begin{theorem}
\label{thm:funct-limit-theor-1}
  Let the FNPP $(N_\alpha(t))_{t\geq 0}$ be defined as in Equation (\ref{eq:23}) and suppose the function $t\mapsto \Lambda(t)$ is regularly varying with index $\beta \in \mathbb{R}$. Moreover let $X_1, X_2, \ldots$ be i.i.d. random variables independent of $N_\alpha$. Assume that the law of $X_1$ is in the domain of attraction of a stable law, i.e. there exist sequences $(a_n)_{n\in \mathbb{N}}$ and $(b_n)_{n\in \mathbb{N}}$ and a stable L\'evy process $(S(t))_{t\geq 0}$ such that for
  \begin{equation*}
    \bar{S}_n(t) := a_n \sum_{k=1}^{\lfloor nt \rfloor} X_k - b_n
  \end{equation*}
it holds that 
\begin{equation}
  \label{eq:8}
  (\bar{S}_n(t))_{t \geq 0} \xrightarrow[n \to \infty]{J_1} (S(t))_{t \geq 0}.
\end{equation}
Then the fractional compound Poisson process $Z(t) := S_{N_\alpha(t)}$ fulfills following limit:
\begin{equation*}
  (c_nZ(nt))_{t\geq 0} \xrightarrow[n \to \infty]{J_1} \left( S\left([Y_\alpha(t)]^\beta\right)\right)_{t \geq 0},
\end{equation*}
where $c_n = a_{\lfloor\Lambda(n)\rfloor}$.
\end{theorem}
\begin{proof}
  The proof follows the technique proposed by \cite{Meerschaert_2004}: By Theorem \ref{thm:regul-vari-fct} we have
\begin{equation*}
   \left(\frac{N_\alpha(t\tau)}{\Lambda(t^\alpha)}\right)_{\tau\geq 0} \xrightarrow[t \rightarrow \infty]{J_1} \left([Y_\alpha(\tau)]^\beta\right)_{\tau\geq 0}.
\end{equation*}
By the independence assumptions we can combine this with (\ref{eq:8}) to get
\begin{equation*}
  \left( a_{\lfloor\Lambda(n^\alpha)\rfloor} S(\Lambda(n^\alpha)) - b_n, [\Lambda(n^\alpha)]^{-1} N_\alpha(nt)\right)_{t\geq 0} \xrightarrow[n\to \infty]{J_1} (S(t), [Y_\alpha(t)]^\beta)_{t\geq 0}
\end{equation*}
in the space $\mathcal{D}([0, \infty), \mathbb{R}\times [0, \infty))$. Note that $([Y_\alpha(t)]^\beta)_{t\geq 0}$ is non-decreasing. Moreover, due to independence the L\'evy processes $(S(t))_{t\geq 0}$ and $(D_\alpha(t))_{t\geq 0}$ do not have simultaneous jumps (for details see \cite{BeckerKern_2004} and more generally \cite{Cont_Tankov_2004}). This allows us to apply Theorem 13.2.4 in \cite{Whitt_2002} to get the assertion by a continuous mapping argument since the composition mapping is continuous in this setting.
\end{proof}
\section{Numerical experiments}
\label{sec:numer-exper}

\subsection{Simulation methods}
\label{sec:simulation-methods}
In the special case of the FHPP, the process is simulated by sampling the waiting times $J_i$ of the overall process $N(Y_\alpha(t))$, which are Mittag-Leffler distributed (see Equation (\ref{eq:16})).
Direct sampling of the waiting times of the FHPP can be done via a transformation formula due to \cite{Kozubowski_Rachev_1999}
\begin{equation*}
  J_1 = - \frac{1}{\lambda} \log(U) \left[ \frac{\sin(\alpha \pi)}{\tan(\alpha\pi V)} - \cos(\alpha \pi)\right]^{1/\alpha},
\end{equation*}
where $U$ and $V$ are two independent random variables uniformly distributed on $[0,1]$. For futher discussion and details on the implementation see \cite{Fulger_2008} and \cite{Germano_2009}.\\

As the above method is not applicable for the FNPP, we draw samples of $Y_\alpha(t)$ first before sampling $N$. The Laplace transform w.r.t. the time variable of $Y_\alpha(t)$
is given by
\begin{equation*}
  \int_0^\infty \e^{-st} h_\alpha(t,x)\,dt = s^{\alpha-1} \exp(-xs^\alpha).
\end{equation*}
We evaluate the density $h_\alpha$ by inverting the Laplace transform numerically using the Post-Widder formula (\cite{Post_1930} and \cite{Widder_1941}): 
\begin{theorem}
  If the integral
  \begin{equation*}
    \bar{f}(s) = \int_0^\infty e^{-su}f(u) \sd u
  \end{equation*}
converges for every $s > \gamma$, then
\begin{equation*}
  f(t) = \lim_{n \rightarrow \infty} \frac{(-1)^n}{n!}\left(\frac{n}{t}\right)^{n+1} \bar{f}^{(n)}\left(\frac{n}{t}\right),
\end{equation*}
for every point $t > 0$ of continuity of $f(t)$ (cf. p. 37  in \cite{Cohen_2007}).
\end{theorem}
This evaluation of the density function allows us to sample $Y_\alpha(t)$ using discrete inversion. 

\subsection{Numerical results}
\label{sec:numerical-results}
Figure \ref{fig:densities} shows the shape and time-evolution of the densities for different values of $\alpha$. As $Y_\alpha$ is an increasing process, the densities spread to the right hand side as time passes.\\
We conducted a small Monte Carlo simulation in order to illustrate the one-dimensional convergence results of Theorem \ref{thm:centr-limit-theor-2} and Theorem \ref{thm:regul-vari-scal}. In Figures \ref{fig:clt01}, \ref{fig:clt02} and \ref{fig:clt03}, we can see that the simulated values for the probability density $x \mapsto \varphi_\alpha(t,x)$ of  $[N(Y_\alpha(t)) - \Lambda(Y_\alpha(t))]/\sqrt{\Lambda(Y_\alpha(t))}$ approximate the density of a standard normal distribution for increasing time $t$. In a similar manner, Figure \ref{fig:scalingLimit} depicts how the probability density function $x \mapsto \phi_\alpha(t,x)$ of $N_\alpha(t)/\Lambda(t^\alpha)$ approximates the density of $(Y_\alpha(t))^\beta$ given in (\ref{eq:10}), where $\Lambda$ has regular variation index $\beta=0.7$.



\begin{figure}
  \centering
  \includegraphics[scale=0.6]{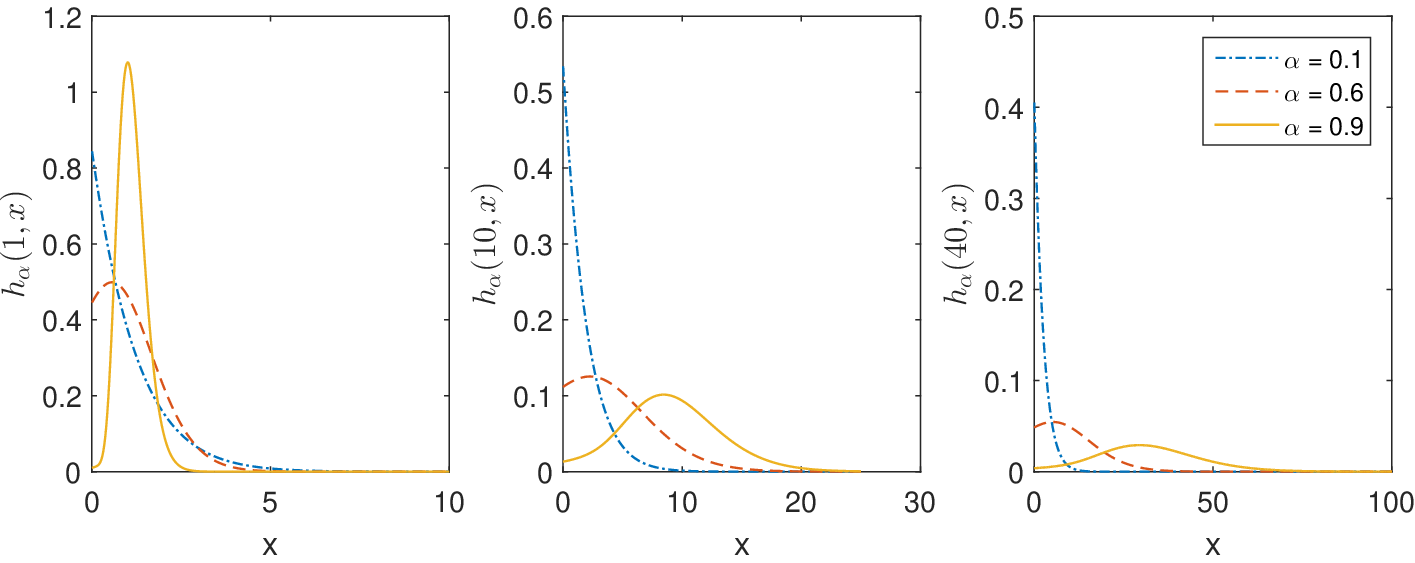}
  \caption{Plots of the probability densities $x \mapsto h_\alpha(t,x)$ of the distribution of the inverse $\alpha$-stable subordinator $Y_\alpha(t)$ for different parameter $\alpha=0.1, 0.6, 0.9$ indicating the time-evolution: the plot on the left is generated for $t=1$, the plot in the middle for $t=10$ and the plot on the right for $t=40$.}
  \label{fig:densities}
\end{figure}

\begin{figure}
  \centering
  \includegraphics[scale=0.6]{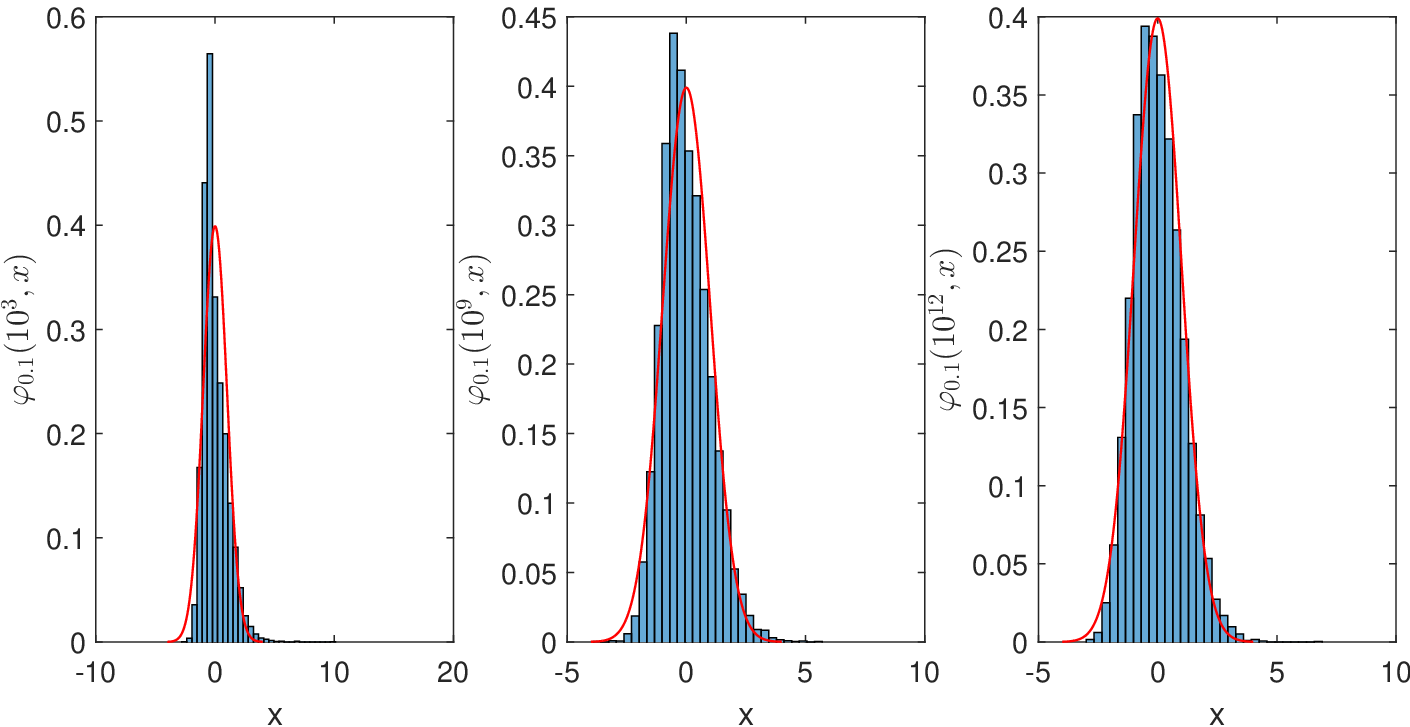}
  \caption{The red line shows the probability density function of the standard normal distribution, the limit distribution according to Theorem \ref{thm:centr-limit-theor-2}. The blue histograms depict samples of size $10^4$ of the right hand side of (\ref{eq:7}) for different times $t=10, 10^9, 10^{12}$ to illustrate convergence to the standard normal distribution.}
  \label{fig:clt01}
\end{figure}

\begin{figure}
  \centering
  \includegraphics[scale=0.6]{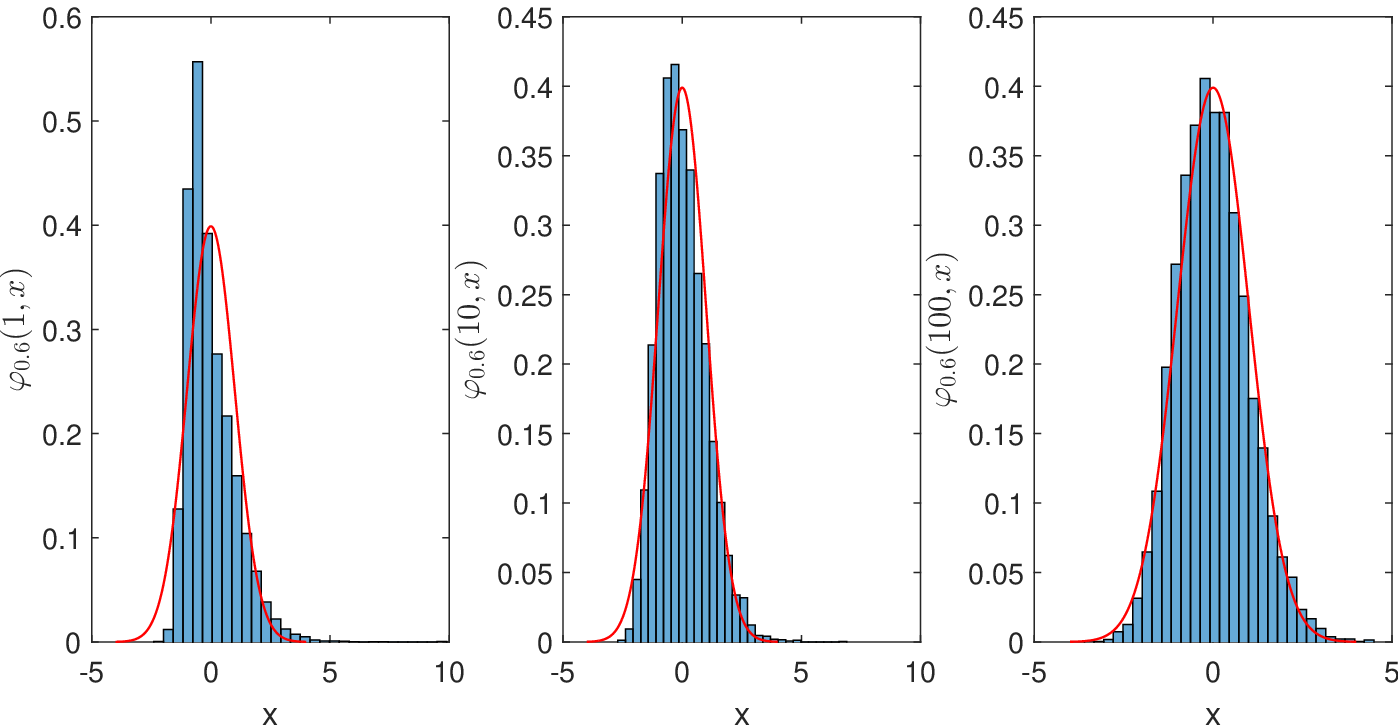}
  \caption{The red line shows the probability density function of the standard normal distribution, the limit distribution according to Theorem \ref{thm:centr-limit-theor-2}. The blue histograms depict samples of size $10^4$ of the right hand side of (\ref{eq:7}) for different times $t=1, 10, 100$ to illustrate convergence to the standard normal distribution.}
  \label{fig:clt02}
\end{figure}

\begin{figure}
  \centering
  \includegraphics[scale=0.6]{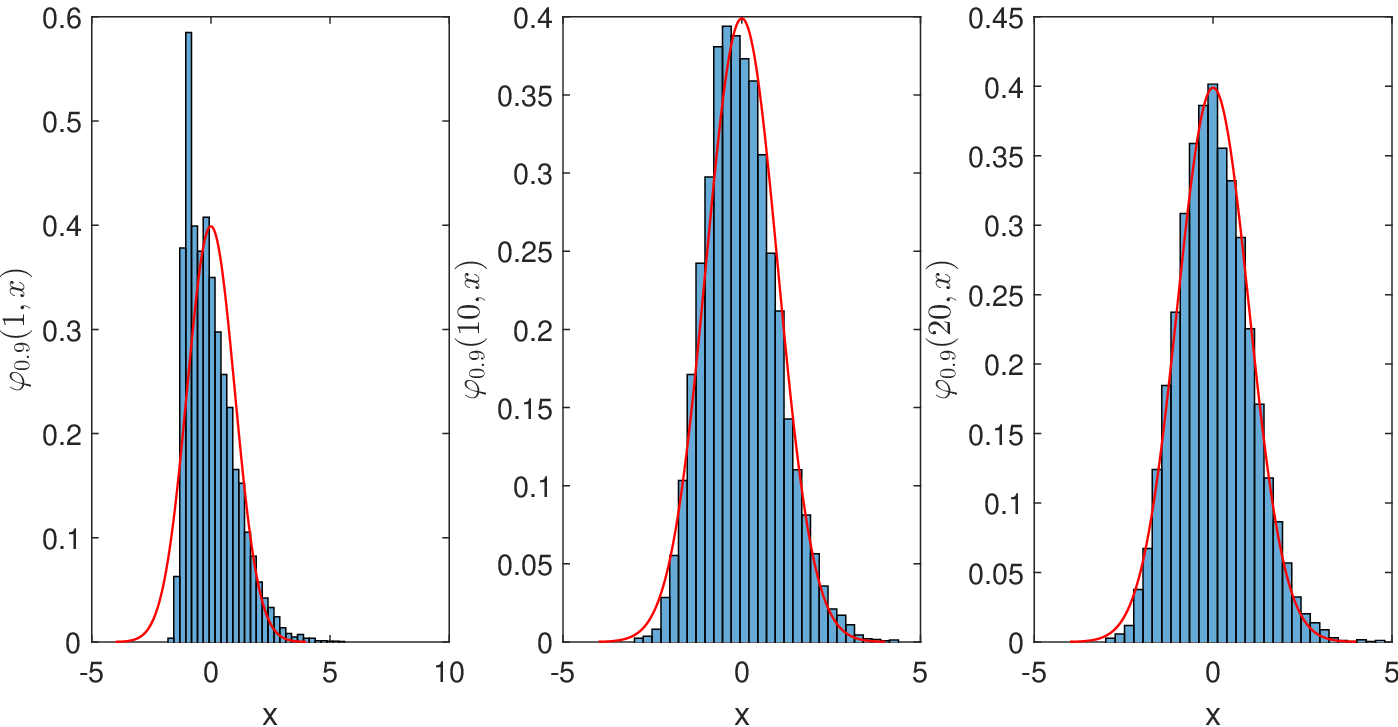}
  \caption{The red line shows the probability density function of the standard normal distribution, the limit distribution according to Theorem \ref{thm:centr-limit-theor-2}. The blue histograms depict samples of size $10^4$ of the right hand side of (\ref{eq:7}) for different times $t=1, 10, 20$ to illustrate convergence to the standard normal distribution.}
  \label{fig:clt03}
\end{figure}

\begin{figure}
  \centering
  \includegraphics[scale=0.6]{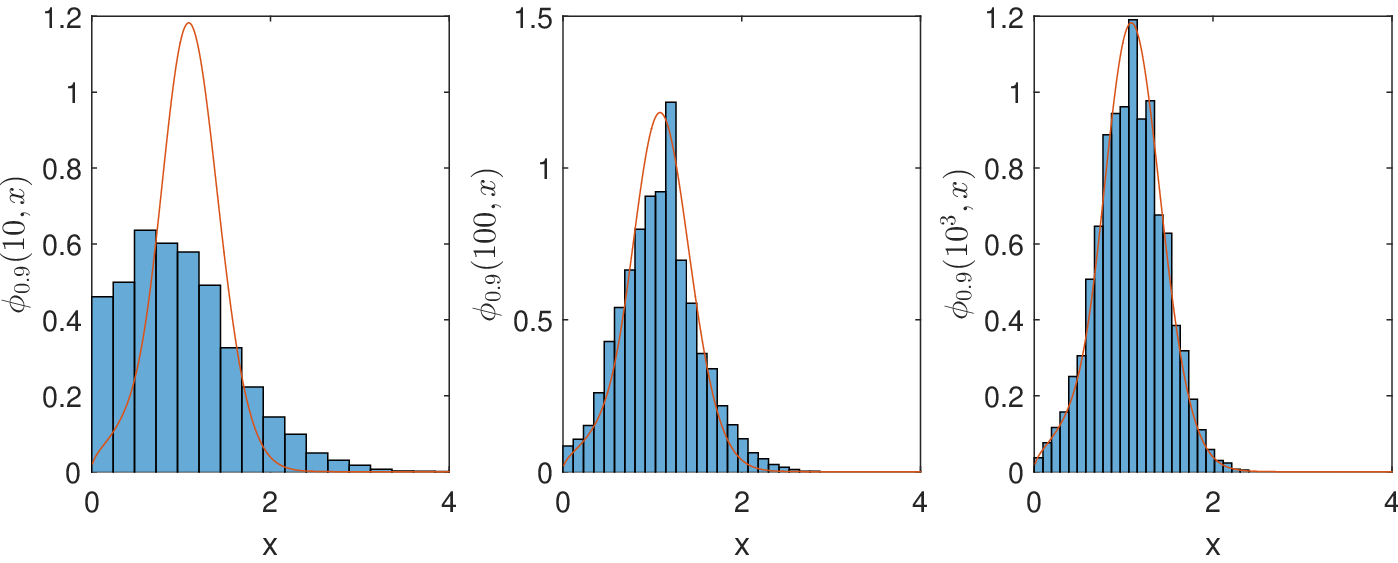}
  \caption{Red line: probability density function $\phi$ of the distribution of the random variable $(Y_{0.9}(1))^{0.7}$, the limit distribution according to Theorem \ref{thm:regul-vari-scal}. The blue histogram is based on $10^4$ samples of the random variables on the right hand side of (\ref{eq:4}) for time points $t=10, 100, 10^3$ to illustrate the convergence result.}
  \label{fig:scalingLimit}
\end{figure}

\section{Summary and outlook}
Due to the non-homogeneous component of the FNPP, it is not surprising that analytical tractability needed to be compromised in order to derive analogous limit theorems. Most noteably, the lack of a renewal representation of the FNPP compared to its homogeneous version lead us to require additional conditions on the underlying filtration structure or rate function $\Lambda$.\\
The result in Proposition \ref{prop:limit-alpha-right} partly answered an open question that followed after Theorem 1 in \cite{Leonenko_2017} concerning the limit $\alpha \rightarrow 1$.\\
Futher research will be directed towards the implications of the limit results for estmation techniques.

\section*{Acknowledgements}
\label{sec:acknowledgement}

N. Leonenko  was supported in particular by Cardiff Incoming Visiting Fellowship Scheme and International Collaboration Seedcorn Fund and Australian Research Council's Discovery Projects funding scheme (project number DP160101366). E. Scalas and M. Trinh were supported by the Strategic Development Fund of the University of Sussex.

  \bibliographystyle{elsarticle-harv} 
  \bibliography{fractionalLTRef}

\begin{thebibliography}{51}
\expandafter\ifx\csname natexlab\endcsname\relax\def\natexlab#1{#1}\fi
\expandafter\ifx\csname url\endcsname\relax
  \def\url#1{\texttt{#1}}\fi
\expandafter\ifx\csname urlprefix\endcsname\relax\def\urlprefix{URL }\fi

\bibitem[{Aletti et~al.(2017)Aletti, Leonenko, and Merzbach}]{Aletti_2016}
Aletti, G., Leonenko, N., Merzbach, E., 2017. Fractional {P}oisson fields and
  martingales, arXiv:1601.08136v2 [math.PR].

\bibitem[{Anscombe(1952)}]{Anscombe_1952}
Anscombe, F.~J., 1952. Large-sample theory of sequential estimation. Proc.
  Cambridge Philos. Soc. 48, 600--607.

\bibitem[{Becker-Kern et~al.(2004)Becker-Kern, Meerschaert, and
  Scheffler}]{BeckerKern_2004}
Becker-Kern, P., Meerschaert, M.~M., Scheffler, H.-P., 2004. Limit theorems for
  coupled continuous time random walks. Ann. Probab. 32~(1B), 730--756.
\newline\urlprefix\url{http://dx.doi.org/10.1214/aop/1079021462}

\bibitem[{Beghin and Orsingher(2009)}]{Beghin_Orsingher_2009}
Beghin, L., Orsingher, E., 2009. Fractional {P}oisson processes and related
  planar random motions. Electron. J. Probab. 14, no. 61, 1790--1827.
\newline\urlprefix\url{http://dx.doi.org/10.1214/EJP.v14-675}

\bibitem[{Beghin and Orsingher(2010)}]{Beghin_Orsingher_2010}
Beghin, L., Orsingher, E., 2010. Poisson-type processes governed by fractional
  and higher-order recursive differential equations. Electron. J. Probab. 15,
  no. 22, 684--709.
\newline\urlprefix\url{http://dx.doi.org/10.1214/EJP.v15-762}

\bibitem[{Biard and Saussereau(2014)}]{Biard_Saussereau_2014}
Biard, R., Saussereau, B., 2014. Fractional {P}oisson process: long-range
  dependence and applications in ruin theory. J. Appl. Probab. 51~(3),
  727--740.
\newline\urlprefix\url{https://doi.org/10.1239/jap/1409932670}

\bibitem[{Biard and Saussereau(2016)}]{Biard_Saussereau_2016}
Biard, R., Saussereau, B., 2016. Correction: ``{F}ractional {P}oisson process:
  long-range dependence and applications in ruin theory'' [ {MR}3256223]. J.
  Appl. Probab. 53~(4), 1271--1272.
\newline\urlprefix\url{https://doi.org/10.1017/jpr.2016.80}

\bibitem[{Bielecki and Rutkowski(2002)}]{Bielecki_2002}
Bielecki, T.~R., Rutkowski, M., 2002. Credit risk: modelling, valuation and
  hedging. Springer Finance. Springer-Verlag, Berlin.

\bibitem[{Bingham(1971)}]{Bingham_1971}
Bingham, N.~H., 1971. Limit theorems for occupation times of {M}arkov
  processes. Z. Wahrscheinlichkeitstheorie und Verw. Gebiete 17, 1--22.

\bibitem[{Bingham et~al.(1989)Bingham, Goldie, and Teugels}]{Bingham_1989}
Bingham, N.~H., Goldie, C.~M., Teugels, J.~L., 1989. Regular variation. Vol.~27
  of Encyclopedia of Mathematics and its Applications. Cambridge University
  Press, Cambridge.

\bibitem[{Br\'emaud(1981)}]{Bremaud_1981}
Br\'emaud, P., 1981. Point processes and queues. Springer-Verlag, New
  York-Berlin, martingale dynamics, Springer Series in Statistics.

\bibitem[{Cahoy et~al.(2010)Cahoy, Uchaikin, and Woyczynski}]{Cahoy_2010}
Cahoy, D.~O., Uchaikin, V.~V., Woyczynski, W.~A., 2010. Parameter estimation
  for fractional {P}oisson processes. J. Statist. Plann. Inference 140~(11),
  3106--3120.
\newline\urlprefix\url{http://dx.doi.org/10.1016/j.jspi.2010.04.016}

\bibitem[{Cohen(2007)}]{Cohen_2007}
Cohen, A.~M., 2007. Numerical methods for {L}aplace transform inversion. Vol.~5
  of Numerical Methods and Algorithms. Springer, New York.

\bibitem[{Cont and Tankov(2004)}]{Cont_Tankov_2004}
Cont, R., Tankov, P., 2004. Financial modelling with jump processes. Chapman \&
  Hall/CRC Financial Mathematics Series. Chapman \& Hall/CRC, Boca Raton, FL.

\bibitem[{Cox(1955)}]{Cox_1955}
Cox, D.~R., 1955. Some statistical methods connected with series of events. J.
  Roy. Statist. Soc. Ser. B. 17, 129--157; discussion, 157--164.
\newline\urlprefix\url{http://links.jstor.org/sici?sici=0035-9246(1955)17:2<129:SSMCWS>2.0.CO;2-9&origin=MSN}

\bibitem[{{Cs\"org\H o} and Fischler(1973)}]{Csorgo_Fischler_1973}
{Cs\"org\H o}, M., Fischler, R., 1973. Some examples and results in the theory
  of mixing and random-sum central limit theorems. Period. Math. Hungar. 3,
  41--57, collection of articles dedicated to the memory of Alfr\'ed R\'enyi,
  II.
\newline\urlprefix\url{http://dx.doi.org/10.1007/BF02018460}

\bibitem[{Daley and Vere-Jones(2003)}]{Daley_2003}
Daley, D.~J., Vere-Jones, D., 2003. An introduction to the theory of point
  processes. {V}ol. {I}, 2nd Edition. Probability and its Applications (New
  York). Springer-Verlag, New York, elementary theory and methods.

\bibitem[{Daley and Vere-Jones(2008)}]{Daley_2008}
Daley, D.~J., Vere-Jones, D., 2008. An introduction to the theory of point
  processes. {V}ol. {II}, 2nd Edition. Probability and its Applications (New
  York). Springer, New York, general theory and structure.
\newline\urlprefix\url{http://dx.doi.org/10.1007/978-0-387-49835-5}

\bibitem[{Di~Crescenzo et~al.(2016)Di~Crescenzo, Martinucci, and
  Meoli}]{DiCrescenzo_2016}
Di~Crescenzo, A., Martinucci, B., Meoli, A., 2016. A fractional counting
  process and its connection with the {P}oisson process. ALEA Lat. Am. J.
  Probab. Math. Stat. 13~(1), 291--307.

\bibitem[{Embrechts and Hofert(2013)}]{Embrechts_Hofert_2013}
Embrechts, P., Hofert, M., 2013. A note on generalized inverses. Math. Methods
  Oper. Res. 77~(3), 423--432.
\newline\urlprefix\url{https://doi.org/10.1007/s00186-013-0436-7}

\bibitem[{Feller(1971)}]{Feller_1971}
Feller, W., 1971. An introduction to probability theory and its applications.
  {V}ol. {II}. Second edition. John Wiley \& Sons, Inc., New
  York-London-Sydney.

\bibitem[{Fulger et~al.(2008)Fulger, Scalas, and Germano}]{Fulger_2008}
Fulger, D., Scalas, E., Germano, G., 2008. Monte carlo simulation of uncoupled
  continuous-time random walks yielding a stochastic solution of the space-time
  fractional diffusion equation. Physical Review E 77~(2), 021122.

\bibitem[{Gergely and Yezhow(1973)}]{Gergely_1973}
Gergely, T., Yezhow, I.~I., 1973. On a construction of ordinary {P}oisson
  processes and their modelling. Z. Wahrscheinlichkeitstheorie und Verw.
  Gebiete 27, 215--232.
\newline\urlprefix\url{https://doi.org/10.1007/BF00535850}

\bibitem[{Gergely and Yezhow(1975)}]{Gergely_1975}
Gergely, T., Yezhow, I.~I., 1975. Asymptotic behaviour of stochastic processes
  modelling an ordinary {P}oisson process. Period. Math. Hungar. 6~(3),
  203--211.
\newline\urlprefix\url{https://doi.org/10.1007/BF02018272}

\bibitem[{Germano et~al.(2009)Germano, Politi, Scalas, and
  Schilling}]{Germano_2009}
Germano, G., Politi, M., Scalas, E., Schilling, R.~L., 2009. Stochastic
  calculus for uncoupled continuous-time random walks. Phys. Rev. E (3) 79~(6),
  066102, 12.
\newline\urlprefix\url{http://dx.doi.org/10.1103/PhysRevE.79.066102}

\bibitem[{Gnedenko and Kovalenko(1968)}]{Gnedenko_1968}
Gnedenko, B.~V., Kovalenko, I.~N., 1968. Introduction to queueing theory.
  Translated from Russian by R. Kondor. Translation edited by D. Louvish.
  Israel Program for Scientific Translations, Jerusalem; Daniel Davey \&\ Co.,
  Inc., Hartford, Conn.

\bibitem[{Grandell(1976)}]{Grandell_1976}
Grandell, J., 1976. Doubly stochastic {P}oisson processes. Lecture Notes in
  Mathematics, Vol. 529. Springer-Verlag, Berlin-New York.

\bibitem[{Grandell(1991)}]{Grandell_1991}
Grandell, J., 1991. Aspects of risk theory. Springer Series in Statistics:
  Probability and its Applications. Springer-Verlag, New York.
\newline\urlprefix\url{http://dx.doi.org/10.1007/978-1-4613-9058-9}

\bibitem[{Gut(2013)}]{Gut_2013}
Gut, A., 2013. Probability: a graduate course, 2nd Edition. Springer Texts in
  Statistics. Springer, New York.
\newline\urlprefix\url{http://dx.doi.org/10.1007/978-1-4614-4708-5}

\bibitem[{Jacod and Shiryaev(2003)}]{Jacod_Shiryaev_2003}
Jacod, J., Shiryaev, A.~N., 2003. Limit theorems for stochastic processes, 2nd
  Edition. Vol. 288 of Grundlehren der Mathematischen Wissenschaften
  [Fundamental Principles of Mathematical Sciences]. Springer-Verlag, Berlin.
\newline\urlprefix\url{http://dx.doi.org/10.1007/978-3-662-05265-5}

\bibitem[{Khintchine(1969)}]{Khintchine_1969}
Khintchine, A.~Y., 1969. Mathematical methods in the theory of queueing. Hafner
  Publishing Co., New York, translated from the Russian by D. M. Andrews and M.
  H. Quenouille, Second edition, with additional notes by Eric Wolman,
  Griffin's Statistical Monographs and Courses, No. 7.

\bibitem[{Kingman(1964)}]{Kingman_1964}
Kingman, J., 1964. On doubly stochastic poisson processes. In: Mathematical
  Proceedings of the Cambridge Philosophical Society. Vol.~60. Cambridge Univ
  Press, pp. 923--930.

\bibitem[{Kozubowski and Rachev(1999)}]{Kozubowski_Rachev_1999}
Kozubowski, T.~J., Rachev, S.~T., 1999. Univariate geometric stable laws. J.
  Comput. Anal. Appl. 1~(2), 177--217.
\newline\urlprefix\url{http://dx.doi.org/10.1023/A:1022629726024}

\bibitem[{Laskin(2003)}]{Laskin_2003}
Laskin, N., 2003. Fractional {P}oisson process. Commun. Nonlinear Sci. Numer.
  Simul. 8~(3-4), 201--213, chaotic transport and complexity in classical and
  quantum dynamics.
\newline\urlprefix\url{http://dx.doi.org/10.1016/S1007-5704(03)00037-6}

\bibitem[{Leonenko and Merzbach(2015)}]{Leonenko_2015}
Leonenko, N., Merzbach, E., 2015. Fractional {P}oisson fields. Methodol.
  Comput. Appl. Probab. 17~(1), 155--168.
\newline\urlprefix\url{http://dx.doi.org/10.1007/s11009-013-9354-7}

\bibitem[{Leonenko et~al.(2017)Leonenko, Scalas, and Trinh}]{Leonenko_2017}
Leonenko, N., Scalas, E., Trinh, M., 2017. The fractional non-homogeneous
  {P}oisson process. Statist. Probab. Lett. 120, 147--156.
\newline\urlprefix\url{http://dx.doi.org/10.1016/j.spl.2016.09.024}

\bibitem[{Mainardi et~al.(2004)Mainardi, Gorenflo, and Scalas}]{Mainardi_2004}
Mainardi, F., Gorenflo, R., Scalas, E., 2004. A fractional generalization of
  the {P}oisson processes. Vietnam J. Math. 32~(Special Issue), 53--64.

\bibitem[{Meerschaert et~al.(2011)Meerschaert, Nane, and
  Vellaisamy}]{Meerschaert_2011}
Meerschaert, M.~M., Nane, E., Vellaisamy, P., 2011. The fractional {P}oisson
  process and the inverse stable subordinator. Electron. J. Probab. 16, no. 59,
  1600--1620.
\newline\urlprefix\url{http://dx.doi.org/10.1214/EJP.v16-920}

\bibitem[{Meerschaert and Scheffler(2004)}]{Meerschaert_2004}
Meerschaert, M.~M., Scheffler, H.-P., 2004. Limit theorems for continuous-time
  random walks with infinite mean waiting times. J. Appl. Probab. 41~(3),
  623--638.

\bibitem[{Meerschaert and Sikorskii(2012)}]{Meerschaert_Sikorskii_2012}
Meerschaert, M.~M., Sikorskii, A., 2012. Stochastic models for fractional
  calculus. Vol.~43 of de Gruyter Studies in Mathematics. Walter de Gruyter \&
  Co., Berlin.

\bibitem[{Meerschaert and Straka(2013)}]{Meerschaert_2013}
Meerschaert, M.~M., Straka, P., 2013. Inverse stable subordinators. Math.
  Model. Nat. Phenom. 8~(2), 1--16.
\newline\urlprefix\url{http://dx.doi.org/10.1051/mmnp/20138201}

\bibitem[{Politi et~al.(2011)Politi, Kaizoji, and Scalas}]{Politi_2011}
Politi, M., Kaizoji, T., Scalas, E., 2011. Full characterisation of the
  fractional poisson process. Europhysics Letters 96~(2).

\bibitem[{Post(1930)}]{Post_1930}
Post, E.~L., 1930. Generalized differentiation. Trans. Amer. Math. Soc. 32~(4),
  723--781.
\newline\urlprefix\url{http://dx.doi.org/10.2307/1989348}

\bibitem[{Richter(1965)}]{Richter_1965}
Richter, W., 1965. {\"U}bertragung von {G}renzaussagen f\"ur {F}olgen von
  zuf\"alligen {G}r\"ossen auf {F}olgen mit zuf\"alligen {I}ndizes. Teor.
  Verojatnost. i Primenen 10, 82--94, this article has appeared in English
  translation [Theor. Probability Appl. 10 (1965), 74–84].

\bibitem[{Samorodnitsky and Taqqu(1994)}]{Samorodnitsky_Taqqu_1994}
Samorodnitsky, G., Taqqu, M.~S., 1994. Stable non-{G}aussian random processes.
  Stochastic Modeling. Chapman \& Hall, New York, stochastic models with
  infinite variance.

\bibitem[{Serfozo(1972{\natexlab{a}})}]{Serfozo_1972a}
Serfozo, R.~F., 1972{\natexlab{a}}. Conditional {P}oisson processes. J. Appl.
  Probability 9, 288--302.

\bibitem[{Serfozo(1972{\natexlab{b}})}]{Serfozo_1972b}
Serfozo, R.~F., 1972{\natexlab{b}}. Processes with conditional stationary
  independent increments. J. Appl. Probability 9, 303--315.

\bibitem[{Watanabe(1964)}]{Watanabe_1964}
Watanabe, S., 1964. On discontinuous additive functionals and {L}\'evy measures
  of a {M}arkov process. Japan. J. Math. 34, 53--70.

\bibitem[{Whitt(2002)}]{Whitt_2002}
Whitt, W., 2002. Stochastic-process limits. Springer Series in Operations
  Research. Springer-Verlag, New York, an introduction to stochastic-process
  limits and their application to queues.

\bibitem[{Widder(1941)}]{Widder_1941}
Widder, D.~V., 1941. The {L}aplace {T}ransform. Princeton Mathematical Series,
  v. 6. Princeton University Press, Princeton, N. J.

\bibitem[{Yannaros(1994)}]{Yannaros_1994}
Yannaros, N., 1994. Weibull renewal processes. Ann. Inst. Statist. Math.
  46~(4), 641--648.
\newline\urlprefix\url{http://dx.doi.org/10.1007/BF00773473}

\end{thebibliography}







\end{document}